\documentclass[12pt,leqno]{amsart}
\usepackage{amsmath,amsfonts,amssymb,amsthm}
\usepackage{graphicx}
\graphicspath{ }
\usepackage{wrapfig}
\usepackage{accents}
\usepackage{caption}
\usepackage{subcaption}
\usepackage{calligra}

\numberwithin{figure}{section}
\newcommand\norm[1]{\left\lVert#1\right\rVert}

\theoremstyle{plain}
\newtheorem{theorem}{Theorem}

\newtheorem{corollary}{Corollary}[theorem]
\theoremstyle{definition}
\newtheorem{definition}{Definition}[section]

\theoremstyle{remark}

\usepackage{mathtools}
\begin{document}
\title{Geodesic Sandwich Theorem with an Application}
\author[A. A. Shaikh, R. P. Agarwal and C. K. Mondal]{Absos Ali Shaikh$^{1\dag}$, Ravi P Agarwal$^2$ and Chandan Kumar Mondal$^{*3}$}

\begin{abstract}
The main goal of the paper is to prove the sandwich theorem for geodesic convex functions in a complete Riemannian manifold. Then by using this theorem we have proved an inequality in a manifold with bounded sectional curvature. Finally, we have shown that the gradient of a convex function is orthogonal to the tangent vector at some point of any geodesic. 
\end{abstract}
\noindent\footnotetext{
$\mathbf{2010}$\hspace{5pt}Mathematics\; Subject\; Classification: 26B25, 39B62, 52A20, 52A30, 52A41, 53C22.\\ 
{Key words and phrases: Convex functions, sandwich theorem, separating convex functions\\}
$^*$The third author greatly acknowledges to The University Grants Commission, Government
of India for the award of Junior Research Fellowship.\\
 $^\dagger$ Corresponding author. }
\maketitle

\section{introduction}
The study of convex function is very important in mathematics, especially, in optimization theory, since many objective functions are convex in a sufficiently small neighborhood of a point which is local minimum. But there are some cases where the objective functions fail to be convex, hence the generalization of convex function becomes necessary. Again to tackle the optimization problem in non-linear space, the notion of convexity in Euclidean space is not sufficient. Hence the concept of convexity has been generalized from Euclidean space to manifold and developed the notion of geodesic convexity, see \cite{AAS11}, \cite{IAA12}, \cite{CK17}. A full discussion about geodesic convexity on a complete Riemannian manifold can be found in \cite{RAP97}, \cite{UDR94}. 
\par The paper is organized as follows. Section 2 deals with some well known facts of Riemannian manifolds and geodesic convexity. In section 3 we have proved the geodesic sandwich theorem as the main result in this paper and as an application of this Theorem we obtain an inequality ( see Theorem 4). In the last section we show that the gradient of a convex function is orthogonal to the tangent of a geodesic in some point.
\section{Preliminaries}
In this section we have discussed some basic facts of a Riemannian manifold $(M,g)$, which will be used throughout this paper (for reference see \cite{JOS11}). Throughout this paper by $M$ we mean a complete Riemannian manifold of dimension $n$ endowed with some positive definite metric $g$ unless otherwise stated. The tangent space at the point $p\in M$ is denoted by $T_pM$ and the tangent bundle is defined by $TM=\cup_{p\in M}T_pM$. The length $l(\gamma)$ of the curve $\gamma:[a,b]\rightarrow M$ is given by
\begin{eqnarray*}
l(\gamma)&=&\int_{a}^{b}\sqrt{g_{\gamma(t)}(\dot{\gamma}(t),\dot{\gamma}(t))}\ dt\\
&=&\int_{a}^{b}\norm {\dot{\gamma}(t)}dt.
\end{eqnarray*}
The curve $\gamma$ is said to be a geodesic if $\nabla_{\dot{\gamma}(t)}\dot{\gamma}(t)=0\ \forall t\in [a,b]$, where $\nabla$ is the Riemannian connection of $g$.
For any point $p\in M$, the exponential map $exp_p:V_p\rightarrow M$ is defined by
$$exp_p(u)=\sigma_u(1),$$
where $\sigma_u$ is a geodesic with $\sigma(0)=p$ and $\dot{\sigma}_u(0)=u$ and $V_p$ is a collection of vectors of $T_pM$ such that for each element $u\in V_p$, the geodesic with initial tangent vector $u$ is defined on $[0,1]$. It can be easily seen that for a geodesic $\sigma$, the norm of a tangent vector is constant, i.e., $\norm {\dot{\gamma}(t)}$ is constant. If the tangent vector of a geodesic is of unit norm, then the geodesic is called normal. If the exponential map exp is defined at all points of $T_pM$ for each $p\in M$, then $M$ is called complete. Hopf-Rinow theorem provides some equivalent cases for the completeness of $M$. Let $x$, $y\in M$. The distance between $p$ and $q$ is defined by
$$d(x,y)=\inf\{l(\gamma):\gamma \text{ be a curve joining }x \text{ and }y\}.$$
A geodesic $\sigma$ joining $x$ and $y$ is called minimal if $l(\sigma)=d(x,y)$. Hopf-Rinow theorem guarantees the existence of minimal geodesic between two points of $M$. A smooth vector field is a smooth function $X:M\rightarrow TM$ such that $\pi\circ X=id_M$, where $\pi:TM\rightarrow M$ is the projection map. The gradient of a smooth function $f:M\rightarrow\mathbb{R}$ at the point $p\in M$ is defined by $\nabla f(p)=g^{ij}\frac{\partial f}{\partial x_{j}}\frac{\partial}{\partial x_i}\mid_p.$
\begin{definition}\cite{UDR94}
 A real valued function $f$ on $M$ is called convex if for every geodesic $\gamma:[a,b]\rightarrow M$, the following inequality holds
\begin{equation*}
f\circ\gamma((1-t)a+tb)\leq (1-t)f\circ\gamma(a)+tf\circ\gamma(b)\quad \forall t\in [0,1].
\end{equation*}
\end{definition}
\section{Geodesic Sandwich theorem}
 Two functions $f,h:M\rightarrow\mathbb{R}$ is said to satisfy the property $(*)$ if the following relation holds
\begin{equation*}
(*)\qquad\qquad\qquad\qquad f(\sigma_{xy}(t))\leq (1-t)h(x)+th(y)\quad \forall x,y\in M,\qquad\qquad\qquad
\end{equation*}
where $\sigma:[0,1]\rightarrow M$ is a geodesic such that $\sigma(0)=x$ and $\sigma(1)=y$.\\
\indent In 1994, Baron et. al. \cite{BMN94} proved the following theorem, which is known as the \textit{sandwich theorem}.
\begin{theorem}[Sandwich Theorem]\cite{BMN94}
Two real valued functions $f$ and $h$ defined on a real interval $I$, satisfy $(*)$ if and only if there is a convex function $k:I\rightarrow\mathbb{R}$ such that
$$f(x)\leq k(x)\leq h(x),\quad \forall x\in I.$$
\end{theorem}
Later this theorem has been extended and generalized in Euclidean space by various authors, see \cite{MM17,LMMQ17,F16, NW95}, but till now no work has been done in manifold. In this section we have extended this result in $M$. The main theorem is as follows:
\begin{theorem}[Geodesic Sandwich Theorem]
Let $f,h:M\rightarrow\mathbb{R}$ be two functions. Then $f$ and $h$ satisfy the property $(*)$ if and only if there exists a geodesic convex function $k:M\rightarrow\mathbb{R}$ such that 
\begin{equation}\label{eq1}
f(x)\leq k(x)\leq h(x),\quad \forall x\in M.
\end{equation}
\end{theorem}
\begin{proof}
Suppose that $f$ and $h$ satisfy the property $(*)$. The epigraph $epi(h)$ of the function $h$ is defined by
$$epi(h)=\{(x,y)\in M\times\mathbb{R}:h(x)\leq y\}.$$
Now $M$ is complete and $\mathbb{R}$ is also complete, hence their product $M\times\mathbb{R}$ is also a complete Riemannian manifold. So, for any two points $p,q\in M\times\mathbb{R}$, there always be a geodesic $\sigma_{pq}:[0,1]\rightarrow M\times\mathbb{R}$ from $p$ to $q$ and each geodesic can be written as
$$\sigma_{pq}(t)=(\sigma_{p_1q_1}(t),(1-t)p_2+tq_2),$$
where $p=(p_1,p_2)$ and $q=(q_1,q_2)$ such that $p_1,q_1\in M$ and $p_2,q_2\in \mathbb{R}$.
Consider the following set
$$E=\bigcup_{u,v\in epi(h)}\{G_h(u,v)\},$$
where $G_h(u,v)=\{\sigma_{uv}(t)\in M\times\mathbb{R}:\sigma_{uv}(0)=u,\ \sigma_{uv}(1)=v,t\in [0,1]\}.$ Now take $(x,y)\in E$, then $(x,y)\in G_h(u,v)$ for some $u,v\in epi(h)$. Also we get
$$h(x_1)\leq y_1 \text{ and }h(x_2)\leq y_2,$$
for $u=(x_1,y_1)$ and $v=(x_2,y_2)$. Since $(x,y)\in G_h(u,v)$, so we get
$$(x,y)=\sigma_{uv}(t_0)=(\sigma_{x_1x_2}(t_0),(1-t_0)y_1+t_0y_2) \text{ for some }t_0\in [0,1].$$
Now let $y_0=(1-t_0)h(x_1)+t_0h(x_2)$. Then
\begin{equation*}
y=(1-t_0)y_1+t_0y_2\geq (1-t_0)h(x_1)+t_0h(x_2)=y_0.
\end{equation*}
Also from the property $(*)$, we get $y\geq y_0\geq f(x)$. Now define the function $k:M\rightarrow\mathbb{R}$ by 
$$k(x)=\inf\{y\in\mathbb{R}:(x,y)\in E\} \text{ for }x\in M.$$ We shall prove that $k$ is a geodesic convex function. Now for each $x_1,x_2\in M$ ad $t\in [0,1]$ we can choose $y_1$ and $y_2$ sufficiently big such that $(x_1,y_1),(x_2,y_2)\in E$, then we have
$$(\sigma_{x_1x_2}(t),(1-t)y_1+ty_2)\in E\ \forall t\in [0,1].$$
Now $k$ satisfies the inequality
$$k(\sigma_{x_1x_2}(t))\leq (1-t)y_1+ty_2.$$ And from the definition of $k$, we get
$$k(\sigma_{x_1x_2}(t))\leq (1-t)k(x_1)+tk(x_2).$$ So $k$ is a geodesic convex function.\\
\indent Conversely suppose that there is a convex function $k:M\rightarrow\mathbb{R}$ which satisfy the inequality $(\ref{eq1})$. Then for a geodesic $\sigma_{xy}:[0,1]\rightarrow M$ with initial point $x$ and final point $y$, we get
\begin{eqnarray*}
f(\sigma_{xy}(t))&\leq& k(\sigma_{xy}(t))\\
&\leq & (1-t)k(x)+tk(y)\\
&\leq & (1-t)h(x)+th(y).
\end{eqnarray*}
Hence we get our result.
\end{proof}
\begin{definition}
Let $f,h:M\rightarrow\mathbb{R}$ be two functions realizing the property $(*)$. Then the convex function $k:M\rightarrow\mathbb{R}$ that satisfies the inequality (\ref{eq1}) is said to be the \textit{separating convex function }for $f$ and $h$.
\end{definition}
The above theorem  guarantees the existence of separating convex function for any two functions, satisfying the property $(*)$. Throughout the paper we consider that the separating convex function is smooth. Since any smooth convex function in a manifold with finite volume is constant \cite{BN69}, so we can state the following result: 
\begin{corollary}
If $M$ possesses finite volume and $f,h:M\rightarrow\mathbb{R}$ are two functions satisfying the property $(*)$, then there is a constant $c\in\mathbb{R}$ such that 
$$f(x)\leq c\leq h(x)\ \forall x\in M.$$ 
\end{corollary}
\begin{corollary}
Let $f,h:M\rightarrow\mathbb{R}$ be two functions satisfying the property $(*)$ and $f>0$. If the manifold $M$ contains a closed geodesic, then $f=h$ along any closed geodesic.
\end{corollary}
\begin{proof}
Since a convex function along a closed geodesic vanishes \cite{BN69}, so from the equation $(\ref{eq1})$ we get our desired result.
\end{proof}
\begin{theorem}\cite{19}\label{th1}
Let $B_R(p)$ be a geodesic ball in $M$. Suppose that the sectional curvature $K_M\leq k$ for some constant $k$ and $R<inj(M,g)$. Then for any real valued smooth function $f$ with $\Delta f\geq 0$ and $f\geq 0$ on $M$,
\begin{equation}
f(p)\leq \frac{1}{V_k(R)}\int_{B_R(p)}f dV,
\end{equation}
where $V_k(R)=\text{Vol}(B_R,g_k)$ is the volume of a ball of radius $R$ in the space form of constant curvature $k$, $dV$ is the volume form and $inj(M,g)$ is the injective radius of $M$.
\end{theorem}
\begin{theorem}
Let $M$ has sectional curvature $K_M\leq s$, for some constant $s$. If $k$ is a separating convex function for $f$ and $h$ and $k\geq 0$, then for any $p\in M$
$$f(p)\leq \frac{3\omega_M}{2V_k}h(r), \text{ for some }r\in \partial B(p,\xi),$$
where $B(p,\xi)$ is the a ball in $M$ with the center $p$ and radius $\xi$ such that $\overline{B(p,\xi)}\subset B(p,R)$ for $R<inj(M)$, $\omega_M$ is the surface area of $\overline{B(p,\xi)}$ and $V_s$ is the volume of the ball $B(\xi)$ of radius $\xi$ in the space form of constant cuvature $s$. 
\end{theorem}
\begin{proof}
Since $k$ is the separating convex function for the real valued functions $f$ and $h$ on $M$, so for any $p\in M$, we get
$$f(p)\leq k(p)\leq h(p).$$
Since $k$ is convex, so $k$ is also subharmonic \cite{GW71}. Then from Theorem \ref{th1}, we get
\begin{eqnarray}
k(p)&\leq& \frac{1}{V_s}\int_{ B(p,\xi)}k dV,\\
&=& \frac{1}{V_s}\int_{\partial  B(p,\xi)}\Big\{\int_{0}^{1}k\circ \sigma_x(t)dt\Big\}dx,
\end{eqnarray}
where $\sigma_x:[0,1]\rightarrow M$ is the minimal geodesic such that $\sigma_x(0)=p$ and $\sigma_x(1)=x$ for each $x\in \partial B(p)$. By using convexity of $f$ we obtain
\begin{eqnarray*}
f(p)\leq k(p)&\leq & \frac{1}{V_s}\int_{\partial  B(p,\xi)}\Big\{\int_{0}^{1}((1-t)k(p)+tk(x))dt\Big\}dx\\
&=& \frac{1}{V_s}\int_{\partial  B(p,\xi)}\Big\{k(p)\int_{0}^{1}(1-t)dt+k(x)\int_{0}^{1}tdt\Big\}dx\\
&=& \frac{1}{2V_s}\int_{\partial  B(p,\xi)}(k(p)+k(x)) dx\\
&=& \frac{\omega_M}{2V_s}k(p)+\frac{1}{V_s}\int_{\partial B(p,\xi)}k(x)dx.
\end{eqnarray*}
Since $k$ is subharmonic and $ \overline{B(p,\xi)}$ lies in the normal neighborhood of $p$, so $\max_{ B(p,\xi)}k=k(r)$ for some $r\in\partial  B(p,\xi).$ Hence from the above inequality we get
\begin{eqnarray*}
f(p)&\leq & k(r)\Big \{\frac{\omega_M}{2V_s}+\frac{1}{V_s}\int_{\partial B(p,\xi)}dx\Big\}\\
&=& k(r)\frac{3\omega_M}{2V_k}\leq \frac{3\omega_M}{2V_k}h(r).
\end{eqnarray*}
\end{proof}

\section{Gradient of a convex function along a geodesic}
\begin{theorem}[Mean value theorem]\cite{AFL05}
Let $f:M\rightarrow\mathbb{R}$ be a differentiable function. Then, for every pair of points $p,q\in M$ and every minimal geodesic path $\sigma:[0,1]\rightarrow M$ joining $p$ and $q$, there exists $t_0\in[0,1]$ such that
$$f(p)-f(q)=d(p,q)df_{\sigma(t_0)}(\sigma'(t_0)),$$
in particular $|f(p)-f(q)|\leq \norm {df(\sigma(t_0))}_{\sigma(t_0)}d(p,q).$
\end{theorem}
\begin{theorem}
Let $f:M\rightarrow\mathbb{R}$ be a convex function and $p\in M$. Then for any $u\in T_pM$ and for any geodesic $\sigma:[0,1]\rightarrow M$ such that $\sigma(0)=p$,  $\sigma'(0)=u$ and $df_p(u)\geq 0$, there exists $t_0\in[0,1]$ such that 
$$df_{\sigma(t_0)}(\sigma'(t_0))=\langle\nabla f(\sigma(t_0)),\sigma'(t_0)\rangle=0.$$ 
\end{theorem}
\begin{proof}
The function $f:M\rightarrow\mathbb{R}$ is convex, hence for any $u\in T_pM$ there exists a minimal geodesic $\sigma:[0,1]\rightarrow M$ such that $\sigma(0)=p$ and $\sigma'(0)=u$ and the following inequality holds
\begin{equation}\label{eq4}
df_p(u)\leq f(q)-f(p),
\end{equation}
where $q=\sigma(1)$.
Now applying mean value theorem for the function $f$ there exists $t_0\in [0,1]$ such that 
\begin{equation}\label{eq2}
f(p)-f(q)=d(p,q)df_{\sigma(t_0)}(\sigma'(t_0)),
\end{equation}
where $d(p,q)=\inf\{L(\sigma):\sigma(0)=p,\ \sigma(1)=q\text{ and }\sigma \text{ is length minimizing}\}$.
Now combining the equations (\ref{eq4}) and (\ref{eq2}) we get
\begin{equation}\label{eq3}
df_p(u)+d(p,q)df_{\sigma(t_0)}(\sigma'(t_0))\leq 0.
\end{equation}
It is given that $df_p(u)\geq 0$, then from (\ref{eq3}) we get $df_{\sigma(t_0)}(\sigma'(t_0))\leq 0$.
Consider a function $\varphi:[0,1]\rightarrow\mathbb{R}$ defined by $$\varphi(t)=df_{\sigma(t)}(\sigma'(t)),\ \forall t\in[0,1].$$
The function $\varphi$ is continuous and $\varphi(0)\geq 0$ and $\varphi(t_0)\leq 0$, hence there exists $\xi\in[0,1]$ such that $\varphi(\xi)=0$. This implies that $$df_{\sigma(\xi)}(\sigma'(\xi))=\langle\nabla f(\sigma(\xi)),\sigma'(\xi)\rangle=0.$$ 
 Hence we get our result. 
\end{proof}

$\bigskip $

$^{1}$Department of Mathematics,

 The University of Burdwan, Golapbag, Burdwan-713104,

 West Bengal, India.

E-mail:aask2003@yahoo.co.in, aashaikh@math.buruniv.ac.in

$\bigskip $

$^{2}$Department of Mathematics, 

Texas A\& M University-Kingsville, Kingsville, Texas 78363-8202
 
 USA.

E-mail:ravi.agarwal@tamuk.edu
$\bigskip $

$^{3}$Department of Mathematics, 

The University of
 Burdwan, Golapbag, Burdwan-713104,
 
 West Bengal, India.

E-mail:chan.alge@gmail.com
$\bigskip $

\end{document}